\documentclass[12pt]{amsart}
\usepackage{amssymb,amsmath,amsthm}  
\usepackage{graphicx}         %
\usepackage{enumerate}
\usepackage{braket}
\usepackage{amsaddr}
\usepackage{amscd}
\usepackage{colordvi}
\usepackage[mathscr]{eucal}
\usepackage{subfigure}
\usepackage{epstopdf}
\usepackage{hyperref}
\usepackage{url}

\numberwithin{equation}{section}
\newtheorem{theorem}{Theorem}[section] 
\newtheorem{definition}[theorem]{Definition}
\newtheorem{lemma}[theorem]{Lemma}

\newtheorem{remark}[theorem]{Remark}

\newcommand{\tu}{\textup}
\newcommand{\bfa}[1]{\boldsymbol{#1}} 			%

\newcommand{\ddiv}{\text{div}}     				%
\usepackage[numbers]{natbib}
\usepackage{filecontents}
\usepackage[latin9]{inputenc}
\usepackage[letterpaper]{geometry}
\usepackage{pifont}
\usepackage{float}
\usepackage{color}
\usepackage{adjustbox}
\usepackage{diagbox}
\usepackage{multirow}
\usepackage{array}
\usepackage{mathrsfs}
\usepackage{tikz}
\usepackage{comment}
\usepackage{csquotes}

\graphicspath{{./Figures/}}

\usepackage{amsfonts}

\usepackage{color}
\usepackage{adjustbox}
\usepackage{diagbox}
\definecolor{black}{rgb}{0,0,0}

\definecolor{red}{rgb}{1,0,0}

\definecolor{blue}{rgb}{0,0,1}

\usepackage{multirow}

\makeatother

\usepackage[english]{babel}

\author[F\MakeLowercase{u}, C\MakeLowercase{hung}, M\MakeLowercase{ai}]
{\Large \Large S\MakeLowercase{hubin} F\MakeLowercase{u}, E\MakeLowercase{ric} C\MakeLowercase{hung}, \Large T\MakeLowercase{ina} M\MakeLowercase{ai}} 
\date{\today}

\title[G\MakeLowercase{eneralized multiscale 
finite element method for a nonlinear elasticity model}]
{\textsf{\LARGE G\MakeLowercase{eneralized 
multiscale finite element method for a strain-limiting nonlinear elasticity model}}}

\begin{document}

\begin{abstract}

In this paper, we consider multiscale methods for nonlinear elasticity.  
In particular, we investigate the Generalized 
Multiscale Finite Element Method (GMsFEM) 
for a strain-limiting elasticity problem.  
Being a special case of the naturally implicit constitutive theory 
of nonlinear elasticity, strain-limiting relation has presented an interesting 
class of material bodies, 
for which strains remain bounded (even infinitesimal) while stresses can become 
arbitrarily large.  
The nonlinearity and material heterogeneities can create multiscale features in the solution,
and multiscale methods are therefore necessary. 
To handle the resulting nonlinear monotone quasilinear elliptic equation, we use linearization based on the Picard iteration.  We consider two types of basis functions, offline and online basis functions, following the general framework of GMsFEM.
The offline basis functions depend nonlinearly on the solution.  Thus, we design an indicator function
and we will recompute the offline basis functions when the indicator function predicts that the material property has significant change during the iterations.
On the other hand, we will use the residual based online basis functions to reduce the error substantially
when updating basis functions is necessary. 
Our numerical results show that the above combination of offline and online basis functions is able to
give accurate solutions with only a few basis functions per each coarse region and updating basis functions in selected iterations.

\end{abstract}

\maketitle


\vskip10pt


\vfill

\noindent Shubin Fu $\cdot$ Eric Chung

\noindent Department of Mathematics, The Chinese University of Hong Kong, Shatin, Hong Kong

\noindent E-mail: shubinfu89@gmail.com (Shubin Fu); tschung@math.cuhk.edu.hk (Eric Chung) \\

\noindent Tina Mai

\noindent Institute of Research and Development, 
Duy Tan University, Da Nang 550000, Vietnam

\noindent E-mail, corresponding author: maitina@duytan.edu.vn; tinagdi@gmail.com


\newpage


\section{Introduction}
Even though multiscale methods for linear equations are positively growing up, 
their applications to nontrivial nonlinear problems are still hard.  
In addition, many nonlinear elastic materials contain multiple scales and 
high contrast.  To overcome the challenge from nonlinearity, the 
main idea is to linearize it in Picard iteration, as employed in \cite{G2} and references therein.  To deal with the 
difficulties from multiple scales and high contrast, 
we apply the GMsFEM (\cite{G1})
for the linear equation at the current iteration.  

The motivation for our nonlinear elasticity problem is a remarkable trend of studying nonlinear responses of materials based on 
the recently developed \textit{implicit constitutive 
theory} (see \cite{Raji03,KRR-ZAMP2007,KRR-MMS2011a,KRR-MMS2011b}).  
As Rajagopal notes, the theory offers a framework for 
establishing nonlinear, 
infinitesimal strain theories for elastic-like (non-dissipative) material 
behavior.  This setting is different from classical
Cauchy and Green approaches for modeling elasticity which, under the condition of 
infinitesimal strains, lead to  traditional linear models.  Moreover, it is 
significant that the implicit constitutive theory provides a firm theoretical 
foundation for modeling fluid and solid mechanics variously, in engineering, physics, and 
chemistry. 

We consider herein a special sub-class of the implicit constitutive theory 
in solids, namely, 
the
\textit{strain-limiting theory}, for which the linearized strain remains bounded, 
even when the stress is very large.  (In the traditional linear model, the stress 
blows up when the 
strain blows up and vice versa.)  Therefore, the strain-limiting theory is helpful for describing the behavior of fracture, 
brittle materials 
near crack tips or notches (e.g.\ crystals), or intensive loads inside the material body or on its boundary.  
Either case leads to stress concentration, regardless of the small gradient of the 
displacement (and thus infinitesimal strain).  Our considering materials are science-non-fiction and physically meaningful.  
These materials can sustain infinite stresses, and 
do not break (because of the boundedness of strains).

  
To solve the multiple scales, 
rather than direct numerical simulations 
on fine scales, 
model reduction methods are proposed in literature to 
reduce the computational expensiveness.  Generally, model reduction techniques 
include upscaling and multiscale approaches.  On coarse grid (which is much 
larger than fine grid), 
upscaling methods involve 
upscaling the media properties based on homogenization to capture macroscopic 
behavior, whereas 
multiscale methods additionally need precomputed multiscale 
basis functions.

Within the framework of multiscale methods, the multiscale finite element method 
(MsFEM), as in \cite{Ms, Msnon}, has effectively proved notable success in a variety of practical 
applications, but it requires scale separation nevertheless.  To overcome this 
requirement, we use the generalized multiscale finite element method (GMsFEM), as in \cite{G1}, 
to systematically construct multiple multiscale basis.  More specifically, in 
the GMsFEM, the computation is divided into two stages: offline and online.  
One constructs, in the offline stage, a small dimension space, which can be used effectively 
in the online stage to construct multiscale basis functions, to solve the problem on coarse grid.  The construction of offline and online spaces is based on the selection of local spectral problems as well as the selection of the snapshot space.  In \cite{gle}, the GMsFEM was applied to handle the linear elasticity problem.  In \cite{yalchin18a}, a model reduction method was introduced to solve nonlinear monotone elliptic equations.  In \cite{gnone}, thanks to \cite{G2} (for handling nonlinearities), the GMsFEM was used to solve nonlinear poroelasticity problems.  

Here, our paper will combine the ideas of Picard iteration and the GMsFEM in \cite{chungres} to solve a strain-limiting nonlinear elasticity problem 
   \cite{A-Mai-Walton, B-Mai-Walton}.
At each Picard iteration, we will either apply the offline GMsFEM or the residual based online adaptive GMsFEM.  In the latter approach, we study the
proposed online basis construction in conjunction with adaptivity (\cite{sara, chungres, chungres1}), which means that online basis functions are added in some
selected regions.  Adaptivity is an important step to obtain an effective local multiscale model reduction as
it is crucial to reduce the cost of online multiscale basis computations.  More specifically, adaptive algorithm allows one to add more basis functions in neighborhoods with more complexity without using a priori information.  Given a sufficient number of initial basis functions in the offline space, our numerical results show that the adaptive
addition of online basis functions substantially decreases the error, accelerates the convergence, and reduces computational cost of the GMsFEM.     

Our strategy is that after some Picard iterations, when the relative change of the permeability coefficient is larger than a given fixed tolerance, we need to update (either offline or online, context-dependently) basis functions.  This updating procedure ends when we obtain desired error, and these new basis functions is kept the same in the next Picard iterations until we need to compute them again.

The next Section contains the mathematical background of our considering 
strain-limiting nonlinear elasticity problem.  Section \ref{pre} is for some 
preliminaries about the GMsFEM, including fine-scale discretization and 
Picard iteration for linearization.  Section \ref{sec4} is devoted to the general idea of GMsFEM, including some existing results regarding offline GMsFEM, for the current nonlinear elasticity problem.  
Section \ref{sec5} is about the existing method of residual based online adaptive GMsFEM for computing online multiscale basis functions, in our context.  In Section \ref{sec6}, some numerical examples will be shown.  The last Section \ref{sec7} is for conclusions.

\section{Formulation of the problem}\label{formulate}
\subsection{Input problem and classical formulation}

Let us consider, in dimension two, 
a nonlinear elastic composite material 
$\Omega = \Omega^1 \times \Omega^2 \in \mathbb{R} \times \mathbb{R}$.  

We assume that the material is being at a static state after 
the action of body forces $\bfa{f}: \Omega \to \mathbb{R}^2$ and traction forces 
$\bfa{G}: \partial \Omega_T 
\to \mathbb{R}^2$.  The boundary of the set $\Omega$ is denoted by $\partial \Omega$, 
which is Lipschitz continuous,  
consisting of two parts $\partial \Omega_T$ and $\partial \Omega_D$, where the displacement 
$\bfa{u}: \Omega \to \mathbb{R}^2$ is given on $\partial \Omega_D$.  
We are considering the strain-limiting 
model of the form (as in \cite{A-Mai-Walton})
\begin{equation}\label{et}
 \bfa{E} = \frac{\bfa{T}}{1 + \beta(\bfa{x})|\bfa{T}|}\,.
\end{equation}
Equivalently,
\begin{equation}\label{te}
 \bfa{T} = \frac{\bfa{E}}{1 - \beta(\bfa{x})|\bfa{E}|}\,.
\end{equation}
In the equations (\ref{et}) and (\ref{te}), $\bfa{T}$ denotes the Cauchy stress 
$\bfa{T}: \Omega \to \mathbb{R}^{2 \times 2}$; whereas, $\bfa{E}$
denotes the classical \textit{linearized} strain tensor,
\begin{equation}\label{e}
\bfa{E} := \frac{1}{2}(\nabla \bfa{u} + \nabla \bfa{u}^{\text{T}})\,.
\end{equation}
We can write 
\[\bfa{E} = \bfa{D}(\bfa{u}) = \bfa{Du} =\nabla_s \bfa{u}\,.\]
Then, by (\ref{te}), it follows that
\begin{equation}\label{tea}
 \bfa{T} = \frac{\bfa{D}(\bfa{u})}{1 - \beta(\bfa{x})|\bfa{D}(\bfa{u})|}\,.
\end{equation}

The strain-limiting parameter function is denoted by $\beta(\bfa{x})$, which 
depends on the position variable $\bfa{x} = (x^1,x^2)$.   
We notice from (\ref{et}) that 
\begin{equation}\label{ebound}
 |\bfa{E}| = \frac{|\bfa{T}|}{1 + \beta(\bfa{x})|\bfa{T}|} < \frac{1}{\beta(\bfa{x})}\,.
\end{equation}
This means that $\displaystyle \frac{1}{\beta(\bfa{x})}$ is the upper-bound on $|\bfa{E}|$, and 
taking sufficiently large $\beta(\bfa{x})$ raises 
the limiting-strain small upper-bound, as desired.  
However, we avoid $\beta(\bfa{x}) \to \infty$.  (If $\beta(\bfa{x}) \to \infty$, then $|\bfa{E}| 
< \displaystyle \frac{1}{\beta(\bfa{x})} 
\to 0$, a contradiction.)  
For the analysis of the problem, $\beta(\bfa{x})$ is assumed to be smooth and have compact range 
$0 < m \leq \beta(\bfa{x}) \leq M$.  Here, $\beta(\bfa{x})$ 
is chosen so that the strong ellipticity condition holds (see \cite{A-Mai-Walton}), that is, $\beta(\bfa{x})$ is large enough, 
to prevent bifurcations arising in numerical simulations.  

\subsection{Function space}
The preliminaries are the same as in \cite{C-G-K}.  Latin indices vary in the set 
$\{1,2\}$.  The space of functions, vector fields in $\mathbb{R}^2$, and $2 \times 2$ matrix fields defined over $\Omega$ are respectively denoted by italic capitals (e.g.\ $L^2(\Omega)$), 
boldface Roman capitals (e.g.\ $\bfa{V}$), 
and special Roman capitals (e.g.\ $\mathbb{S}$).  
The space of symmetric matrices of order 2 is denoted by $\mathbb{S}^2$.  The subscript $s$ 
appended to a special Roman capital denotes a space of symmetric matrix fields.

Our considering space is $\bfa{V}: = \bfa{H}_0^1(\Omega) = \bfa{W}_0^{1,2}(\Omega)$.  
However, the techniques here can be used 
in more general space $\bfa{W}_0^{1,p}(\Omega)$, where $2\leq p < \infty$.  
The reason we consider the space 
$\bfa{W}_0^{1,2}(\Omega)$ is that we can characterize displacements that 
vanish on the boundary $\partial \Omega$ of $\Omega$.
The dual space 
(also called the adjoint space), which consists of continuous linear functionals on $\bfa{H}_0^1(\Omega)$, is denoted by 
$\bfa{H}^{-1}(\Omega)$, and 
the value of a functional $\bfa{b} \in \bfa{H}^{-1}(\Omega)$ at a point 
$\bfa{v} \in \bfa{H}_0^1(\Omega)$ is 
denoted by $\langle \bfa{b},\bfa{v} \rangle$.  

The Sobolev norm $\| \cdot \|_{\bfa{W}_0^{1,2}(\Omega)}$ is of the form 
\[\|\bfa{v}\|_{\bfa{W}_0^{1,2}(\Omega)} = (\|\bfa{v}\|^2_{\bfa{L}^2(\Omega)} + 
\|\nabla \bfa{v}\|^2_{\mathbb{L}^2(\Omega)})^{\frac{1}{2}}\,.\]
Here, $\| \bfa{v}\|_{\bfa{L}^2(\Omega)}:= \| | \bfa{v}| \|_{\bfa{L}^2(\Omega)}\,,$ where 
$| \bfa{v}|$ denotes the Euclidean norm of the 2-component vector-valued function 
$ \bfa{v}$; and 
$\| \nabla \bfa{v}\|_{\mathbb{L}^2(\Omega)}:= \| | \nabla \bfa{v}| \|_{\mathbb{L}^2(\Omega)}\,,$ where 
$| \nabla \bfa{v}|$ denotes the Frobenius norm of the $2 \times 2$ matrix $\nabla \bfa{v}$.  
We recall that the Frobenius norm on $\mathbb{L}^2(\Omega)$ is defined by 
$| \bfa{X} |^2 : = \bfa{X} \cdot \bfa{X} = \text{tr}(\bfa{X}^{\text{T}} \bfa{X})\,.$

The dual norm to $\| \cdot \|_{\bfa{H}_0^1(\Omega)}$ is $\| \cdot \|_{\bfa{H}^{-1}(\Omega)}$, i.e.,
\[\| \bfa{b} \|_{\bfa{H}^{-1}(\Omega)} = \sup_{\bfa{v} \in \bfa{H}_0^1(\Omega)}
\frac{|\langle \bfa{b},\bfa{v} \rangle|}{\|\bfa{v}\|_{\bfa{H}_0^1(\Omega)}}\,.\]

Let $\Omega$ be a bounded, simply connected, open, Lipschitz, convex domain of $\mathbb{R}^2$.  Let 
\begin{equation}\label{H1*}
 \bfa{f} \in \bfa{H}^{1}_*(\Omega)= 
\left \{\bfa{g} \in \bfa{H}^1(\Omega) \biggr | \int_{\Omega} \bfa{g} \, dx = \bfa{0}\right \} 
\subset \bfa{L}^2(\Omega) \subsetneq \bfa{H}^{-1}(\Omega)\,.
\end{equation}  
be bounded in $\bfa{L}^2(\Omega)$.
We consider the following problem:  
Find $\bfa{u} \in \bfa{H}^1(\Omega)$ and $\bfa{T} \in \mathbb{L}^1(\Omega)$ such that
\begin{align}\label{form1a}
\begin{split}
 -\textup{\ddiv} (\bfa{T}) &= \bfa{f}  \quad \text{in } \Omega \,,\\
  \bfa{Du} &= \frac{\bfa{T}}{1 + \beta(\bfa{x})|\bfa{T}|} \quad \text{in } 
 \Omega \,,\\
 \bfa{u} &= \bfa{0} \quad \text{on } \partial \Omega_D\,,\\
 \bfa{Tn} &= \bfa{G} \quad \text{on } \partial \Omega_T\,,
 \end{split}
 \end{align}
where $\bfa{n}$ stands for the outer unit normal vector to the boundary of $\Omega$. 



Benefiting from the notations in \cite{B-M-S}, we will write
\[\bfa{S} \text{ as } \bfa{T} \qquad \text{and} \qquad \bfa{D}(\bfa{u})= \bfa{Du} 
\text{ as } \bfa{E} = \bfa{E}(\bfa{u})\,.\] 

The considering model (\ref{et}) is compatible with the laws of thermodynamics 
\cite{KRR-ARS-PRSA2007,Rajagopal493}, which means that the class of materials are non-dissipative 
and are elastic. 

We assume $\partial \Omega_T = \emptyset$.  Using (\ref{form1a}), 
   we write the considering formulation in the form of displacement problem:  
   Find $\bfa{u} \in \bfa{H}_0^1(\Omega)$ such that
   \begin{align}
 -\textup{\ddiv} \left(\frac{\bfa{D}(\bfa{u})}{1 - \beta(\bfa{x})|\bfa{D}(\bfa{u})|}\right) 
 &= \bfa{f} \quad \text{in } 
 \Omega \,,\label{form3} \\ 
 \bfa{u}&= \bfa{0}
 \quad \text{on } \partial \Omega\,.\label{D}
 \end{align}
 We denote 
 \begin{equation}\label{form4}
  \kappa(\bfa{x}, |\bfa{D}(\bfa{u})|)=\frac{1}{1 - 
  \beta(\bfa{x})|\bfa{D}(\bfa{u})|}\,, \quad
  \bfa{a}(\bfa{x},\bfa{D}(\bfa{u})) = 
  \kappa(\bfa{x}, |\bfa{D}(\bfa{u})|)\bfa{D}(\bfa{u})\,,
 \end{equation}
 in which $\bfa{u}(\bfa{x}) \in 
   \bfa{W}_0^{1,2}(\Omega)$.  In this setting, 
   $\bfa{a(\bfa{x},\bfa{\xi})} \in 
   \mathbb{L}^1(\Omega)$, $\bfa{\xi} \in \mathbb{L}^{\infty}(\Omega)$, as 
   in \cite{Beck2017}.
   
\subsection{Existence and uniqueness}
Notice from (\ref{ebound}) that $\displaystyle \frac{1}{\beta(\bfa{x})}$ 
is the upper-bound on $|\bfa{D}(\bfa{u})|$, and 
$0 < m \leq \beta(\bfa{x}) \leq M$ 
such that 
\begin{equation}\label{ebound2}
 0 \leq |\bfa{D}(\bfa{u})| < \frac{1}{M}\leq
 \displaystyle \frac{1}{\beta(\bfa{x})} \leq \frac{1}{m} \,.
\end{equation}

We define 
\begin{equation}\label{Fz}
\bfa{F}(\bfa{\xi})= \frac{\bfa{\xi}}{1-\beta(\bfa{x})|\bfa{\xi}|}\,,
\end{equation}
where $\bfa{\xi} \in \mathbb{L}^{\infty}(\Omega)$, and
\begin{equation}\label{zbound}
 0 \leq |\bfa{\xi}| < \frac{1}{M} \leq \frac{1}{\beta(\bfa{x})}\,.
\end{equation}  

With these facts and thanks to \cite{B-M-S}, 
we derive the following results, which were also stated in 
\cite{BMRS14} (p.\ 19).

\begin{lemma}\label{lem1}
Let 
\begin{equation}\label{calZ}
 \mathcal{Z}:= \left \{ \bfa{\zeta} \in \mathbb{R}^{2 \times 2} \; \biggr | \; 
 0 \leq |\bfa{\zeta}| <\dfrac{1}{M}  \right \}\,.
\end{equation}
 For any $\bfa{\xi}\in \mathcal{Z}$ such that $0 \leq |\bfa{\xi}|  
< \dfrac{1}{M}$, consider the mapping 
 \[\bfa{\xi} \in \mathcal{Z} \mapsto 
 \bfa{F}(\bfa{\xi}): = \frac{\bfa{\xi}}{1-\beta(\bfa{x}) |\bfa{\xi}|} 
 \in \mathbb{R}^{2 \times 2}\,.\]
 Then, for each $\bfa{\xi}_1, \bfa{\xi}_2 \in \mathcal{Z}$, we have
 \begin{align}
  |\bfa{F}(\bfa{\xi}_1) - \bfa{F}(\bfa{\xi}_2)| 
  &\leq \frac{|\bfa{\xi}_1-\bfa{\xi}_2|} 
  {(1- \beta(\bfa{x})(|\bfa{\xi}_1| + |\bfa{\xi}_2|))^2}\,, \label{cont1} \\
  (\bfa{F}(\bfa{\xi}_1) - \bfa{F}(\bfa{\xi}_2)) \cdot (\bfa{\xi}_1-\bfa{\xi}_2) 
  &\geq |\bfa{\xi}_1-\bfa{\xi}_2|^2\,. \label{mono1}
 \end{align}
\end{lemma}

\begin{proof}
We present here the proof in details, thanks to \cite{B-M-S}.  
Notice first that 
\begin{align*}
\bfa{F}(\bfa{\xi}_1) - \bfa{F}(\bfa{\xi}_2) 
& = \int_0^1 \frac{d}{dt} \bfa{F}(t\bfa{\xi}_1 + (1-t)\bfa{\xi}_2)dt \\
&= \int_0^1 \frac{d}{dt}\left(\frac{t\bfa{\xi}_1 + (1-t)\bfa{\xi}_2}
{1-\beta(\bfa{x}) |t\bfa{\xi}_1 + (1-t)\bfa{\xi}_2|}
\right)dt\\
& =\int_0^1 \frac{\bfa{\xi}_1-\bfa{\xi}_2}
{(1-\beta(\bfa{x}) |t\bfa{\xi}_1 + (1-t)\bfa{\xi}_2|)^2} dt\,.
\end{align*}

 For the proof of (\ref{cont1}), we observe that 
 \begin{align*}
  |t\bfa{\xi}_1+ (1-t)\bfa{\xi}_2| \leq \textup{max}\{|\bfa{\xi}_1|, |\bfa{\xi}_2|\} 
  \leq |\bfa{\xi}_1| + |\bfa{\xi}_2| \,.
 \end{align*} 
Therefore,
 \begin{align*}
  |\bfa{F}(\bfa{\xi}_1) - \bfa{F}(\bfa{\xi}_2)| 
  &\leq \frac{|\bfa{\xi}_1-\bfa{\xi}_2|}
  {(1- \beta(\bfa{x})(|\bfa{\xi}_1| + |\bfa{\xi}_2|))^2} \,.
 \end{align*}
Then (\ref{cont1}) follows.
 
For (\ref{mono1}), we notice that 
$(1- \beta(\bfa{x}) |t\bfa{\xi}_1+(1-t)\bfa{\xi}_2|)^2 \leq 1$.  
Thus,
\begin{align*}
 (\bfa{F}(\bfa{\xi}_1)-\bfa{F}(\bfa{\xi}_2)) \cdot (\bfa{\xi}_1-\bfa{\xi}_2) = 
 \int_0^1 \frac{|\bfa{\xi}_1-\bfa{\xi}_2|^2}
 {(1- \beta(\bfa{x}) |t\bfa{\xi}_1+(1-t)\bfa{\xi}_2|)^2} \, dt 
 \geq |\bfa{\xi}_1-\bfa{\xi}_2|^2\,.
\end{align*}
\end{proof}

\begin{remark}
 The condition (\ref{mono1}) also implies that $\bfa{F}(\bfa{\xi})$ is a monotone operator in 
 $\bfa{\xi}$.
\end{remark}

\begin{remark}\label{caZ}
For the rest of the paper, without confusion, we will use the condition 
$\bfa{\xi} \in \mathbb{L}^{\infty}(\Omega)$ with the meaning that 
$\bfa{\xi} \in \mathcal{Z}' = \left \{ \bfa{\zeta} \in \mathbb{L}^{\infty}(\Omega) \; \biggr | \; 
 0 \leq |\bfa{\zeta}| <\dfrac{1}{M}  \right \}$.
\end{remark}


\subsubsection{Weak formulation}  
Let 
\begin{equation}\label{calU}
 \mathcal{U} = \{ \bfa{w} \in \bfa{H}^1(\Omega) \; | 
 \; \bfa{Dw} \in \mathcal{Z}'\}\,,
\end{equation}
with the given $\mathcal{Z}'$ in Remark \ref{caZ}. 
\begin{remark}\label{caU}
For the rest of the paper, without confusion, we will use the condition 
$\bfa{u}, \bfa{v} \in \bfa{H}_0^1(\Omega) \text{ or } \bfa{H}^1(\Omega)$ 
(context-dependently)
with the meaning that 
$\bfa{u}, \bfa{v} \in \mathcal{U}$.
\end{remark}

Now, for $\bfa{u} \in \bfa{V} =\bfa{H}_0^1(\Omega)$, we multiply equation (\ref{form3}) 
by $\bfa{v} \in \bfa{V}$ and integrate the equation 
with respect to $\bfa{x}$ over $\Omega$.  
Integrating the first term by parts and using the condition 
$\bfa{v} = \bfa{0}$ on $\partial \Omega$, we obtain
\begin{equation}\label{w8.2}
 \int_{\Omega} \bfa{a}  (\bfa{x},\bfa{D}\bfa{u} ) \cdot \bfa{D}\bfa{v} \, dx = 
 \int_{\Omega} \bfa{f} \cdot \bfa{v} \, dx 
 \,, \quad \forall \bfa{v} \in \bfa{V} \,.
\end{equation}
By the weak (often called generalized) formulation of the boundary value problem 
(\ref{form3})-(\ref{D}), we interpret the problem as follows:
\begin{equation}\label{w8.3}
\text{ Find } (\bfa{u}, \bfa{Du})  \in \bfa{V} \times 
\mathbb{L}^{\infty}(\Omega), 
\text{that is, find } 
\bfa{u} \in \bfa{V}
\text{ such that } (\ref{w8.2}) \text{ holds for each } \bfa{v} \in \bfa{V}\,.
\end{equation}

\subsubsection{Existence and uniqueness}
In \cite{B-M-S}, the existence and uniqueness of weak solution $(\bfa{u}, \bfa{T})$ 
for (\ref{form1a}) have been proved.  Also, see \cite{iib}, for further reference. 

Similarly, in our paper, we consider the problem:  Find $(\bfa{u}, \bfa{T}) 
\in  \bfa{H}_0^1(\Omega) \times \mathbb{L}^1(\Omega) $ such that
\begin{align}
\int_{\Omega} \bfa{T} \cdot \bfa{D}(\bfa{w}) \, dx &= \int_{\Omega} \bfa{f} \cdot \bfa{w} \, dx 
\quad  \forall \bfa{w} \in \bfa{H}_0^1(\Omega) \,, \label{w4}\\
 \bfa{D}(\bfa{u}) &= \frac{\bfa{T}}{1 + \beta(\bfa{x})|\bfa{T}|}  \quad \text{in } 
 \Omega \,,\label{f4} \\ 
 \bfa{u}&= \bfa{0} \quad \text{on } \partial \Omega\,.\label{D4}
 \end{align}
 This problem is equivalent to problem (\ref{w8.3}).

As noticed in \cite{BMRS14} (Section 4.3), 
the identity (\ref{et}) can be equivalently
rewritten as
\[\bfa{T} = \bfa{a}(\bfa{x}, \bfa{Du}) = \frac{\bfa{Du}}{1 - \beta(\bfa{x}) |\bfa{Du}|}\,,\]
where $\bfa{a}$ is a uniformly monotone operator (\ref{mono1}) 
with at most linear growth at infinity (\ref{cont1}).  Hence, the existence and uniqueness of the solution 
$\bfa{T} \in \mathbb{L}^1(\Omega)$ and $\bfa{u} \in \bfa{H}_0^{1}(\Omega)$ (or, in higher regularity, $\bfa{T} \in \mathbb{L}^2(\Omega)$ and $\bfa{u} \in \bfa{W}_0^{1,2}(\Omega)$) to 
(\ref{w4}) - (\ref{D4}), or $\bfa{u} \in \bfa{W}_0^{1,2}(\Omega)$ to (\ref{w8.2}), 
is guaranteed by \cite{Beck2017}.  

In the case 
with $\beta(\bfa{x})$ in (\ref{form1a}), these results 
are still valid, arriving from
similar argument as in \cite{B-M-S}, thanks to Lemma \ref{lem1}.


\vspace{10pt}


\section{Fine-scale discretization and Picard iteration for linearization}
\label{pre}
The solution $\bfa{u} \in \bfa{V}$ to (\ref{form3}) satisfies
\begin{equation}\label{2cem}
 q(\bfa{u},\bfa{v}) = (\bfa{f},\bfa{v}), \quad \forall \bfa{v} \in \bfa{V}\,,
\end{equation}
where
\begin{align}\label{3cem}
 q(\bfa{u}, \bfa{v}) = \int_{\Omega} 
 \bfa{a} (\bfa{x}, \bfa{Du}) \cdot \bfa{Dv} \, dx, 
 \quad (\bfa{f},\bfa{v}) = \int_{\Omega} \bfa{f} \cdot \bfa{v} \, dx\,.
\end{align}

Starting with an initial guess $\bfa{u}^0 = \bfa{0}$, to solve equation (\ref{form3}), we will linearize it by Picard iteration, that is, we solve
\begin{align}
 -\ddiv (\kappa(\bfa{x}, |\bfa{D}( \bfa{u}^n)|) 
 \bfa{D}( \bfa{u}^{n+1})) &= \bfa{f} \quad \text{in } \Omega\,, \label{4cem} \\
 \bfa{u}^{n+1} &= \bfa{0} \quad \text{on } \partial \Omega\,,
\end{align}
where superscripts involving $n\text{ } (\geq 0)$ denote respective iteration levels.  

To discretize (\ref{4cem}), we next introduce the notion of fine and coarse grids.  Let $\mathcal{T}^H$ be a conforming partition of the domain $\Omega$.  We call $H$ the coarse mesh size and $\mathcal{T}^H$ the coarse grid.  Each element of $\mathcal{T}^H$ is called a coarse grid block (patch).  
We denote by $N_v$ the total number of interior vertices of 
$\mathcal{T}^H$ and $N$ the total number of coarse blocks.  Let $\{\bfa{x}_i\}^{N_v}_{i=1}$ be the set of vertices in $\mathcal{T}^H$ and $w_i= \cup \{K_j \in \mathcal{T}^H \; | \; \bfa{x}_i \in \bar{K_j}\}$ be the neighborhood of the node $\bfa{x}_i$.  The conforming refinement of the triangulation $\mathcal{T}^H$ is denoted by $\mathcal{T}_h$, which is called the fine grid, where $h > 0$ is the fine mesh size.  We assume that $h$ is very small so that the fine-scale solution $\bfa{u}_h$ (to be founded in the next paragraph) is sufficiently close to the exact solution.  The main goal of this paper is to find a multiscale solution $\bfa{u}_{\tu{ms}}$ which is a good approximation of the fine-scale solution 
$\bfa{u}_h$.  
This is the reason why the GMsFEM is used to obtain the multiscale solution $\bfa{u}_{\tu{ms}}$. 

On the fine grid $\mathcal{T}_h$, we will approximate the solution of (\ref{2cem}), denoted by $\bfa{u}_h$ (or $\bfa{u}$ for brevity).  To fix the notation, we will use Picard iteration and the first-order (linear) finite elements for the computation of the fine-scale solution $\bfa{u}_h$.  In particular, we let $\bfa{V}_h\text{ } (\subset \bfa{V} = \bfa{H}_0^1(\Omega))$ be the first-order Galerkin finite element basis space with respect to the fine grid $\mathcal{T}_h$.  
Toward presenting the details of the Picard iteration algorithm, we define the bilinear form 
$a(\cdot, \cdot ; \cdot)$
\begin{equation}\label{biform}
 a(\bfa{u},\bfa{v}; |\bfa{Dw}|) = \int_{\Omega} \kappa
 (\bfa{x},|\bfa{Dw}|) 
 (\bfa{D} \bfa{u} \cdot \bfa{D} \bfa{v}) dx\,,
\end{equation}
and the functional $J(\cdot)$
\begin{equation}\label{func}
 J(\bfa{v}) = \int_{\Omega} \bfa{f} \cdot \bfa{v} dx\,.
\end{equation}

Given $\bfa{u}_h^n$, 
the next approximation $\bfa{u}_h^{n+1}$ is 
the solution of the linear elliptic equation 
\begin{equation}\label{leq1}
 a(\bfa{u}_h^{n+1},\bfa{v}; 
 |\bfa{D}(\bfa{u}_h^n)|) = J(\bfa{v}), \quad \forall \bfa{v} 
 \in \bfa{V}_h\,.
\end{equation}
This is an approximation of the linear equation 
\begin{equation}\label{leq2}
 -\ddiv (\kappa(\bfa{x},|\bfa{D}( \bfa{u}_h^n)|) 
 \bfa{D}( \bfa{u}_h^{n+1})) = \bfa{f}\,.
\end{equation}  

We reformulate the iteration (\ref{leq1}) in a matrix form.  
That is, we define $\bfa{A}^n$ by 
\begin{equation}\label{matrix1}
 a(\bfa{w},\bfa{v};|\bfa{D}( \bfa{u}_h^n)|) = 
 \bfa{v}^{\text{T}}\bfa{A}^n \bfa{w} \quad \forall \bfa{v},\bfa{w} 
 \in \bfa{V}_h\,.
\end{equation}
and define vector $\bfa{b}$ by 
\begin{equation}\label{vector}
 J(\bfa{v}) = \bfa{v}^{\text{T}} \bfa{b}, \quad \forall \bfa{v} 
 \in \bfa{V}_h\,.
\end{equation}
Then, in $\bfa{V}_h$, the equation (\ref{leq1}) can be rewritten in the following matrix form:
\begin{equation}\label{matrix2}
 \bfa{A}^n \bfa{u}_h^{n+1} = \bfa{b}\,.
\end{equation}

Furthermore, at the $(n+1)$-th Picard iteration, we can solve (\ref{matrix2}) for the multiscale solution $\bfa{u}^{n+1}_{\tu{ms}} \in \bfa{V}^n_{\tu{ms}}$ by using the GMsFEM (to be discussed in the next Sections \ref{sec4} and \ref{sec5}), with multiscale basis functions in $\bfa{V}^n_{\tu{ms}}$ computed for $|\bfa{D} \bfa{u}_{\tu{ms}}^n|$ in each coarse region $w_i, i = 1, \cdots, N$.  

Each of $\bfa{u}_h$ and $\bfa{u}_{\tu{ms}}$ is computed in a separate Picard iteration procedure, whose termination criterion is that the relative $\bfa{L}^2$ difference is less than $\delta_0$, which can be found in Subsection \ref{pu} and Section \ref{sec6} 
($\delta_0=10^{-7}$).


\section{GMsFEM for nonlinear elasticity problem}\label{sec4}
\subsection{Overview}
We will construct the offline and online spaces.  Being motivated by \cite{gnone}, we will concentrate on the effects of the nonlinearities.  From the linearized formulation 
(\ref{leq2}), we can define offline and online basis functions following the general framework of GMsFEM. 

Given $\bfa{u}^n$ (which can stand for either $\bfa{u}_h^n$ or $\bfa{u}^n_{\tu{ms}}$, context-dependently), at the current $(n+1)$-th Picard iteration, we will obtain the fine-scale solution $\bfa{u}_h^{n+1} \in \bfa{V}_h$ by solving the variational problem
\begin{equation}\label{f3}
 a_n(\bfa{u}_h^{n+1}, \bfa{v}) = (\bfa{f}, \bfa{v}), \quad 
\forall \bfa{v} \in \bfa{V}_h\,,
\end{equation}
where
\begin{equation}\label{biform1}
 a_n(\bfa{w},\bfa{v}) = \int_{\Omega} \kappa
 (\bfa{x},|\bfa{Du}^n|) 
 (\bfa{D} \bfa{w} \cdot \bfa{D} \bfa{v}) dx\,.
\end{equation}

At the $n$-th Picard iteration, we equip the space $\bfa{V}_h$ with the energy norm $\|\bfa{v}\|^2_{\bfa{V}_h} = a_n(\bfa{v}, \bfa{v})$.

\subsection{General idea of GMsFEM}
For details of GMsFEM, we refer the readers to \cite{G2, chungres1, chungres, chung2016adaptive}.  In this paper, at the current $n$-th Picard iteration, we will consider the continuous Galerkin (CG) formulation, having a similar form to the fine-scale problem (\ref{f3}).  

First, we start with constructing snapshot functions.  Then, by solving a class of specific spectral problems in that snapshot space, for each coarse node $\bfa{x}_i$, we will obtain a set of multiscale basis functions $\{ \bfa{\psi}_k^i \,|\, k = 1, 2, \cdots, l_i \}$, such that each $\bfa{\psi}_k^i = \bfa{\psi}_k^{w_i}$ is supported on the coarse neighborhood $w_i$.  Furthermore, the basis functions satisfy a partition of unity property, that is, there exist coefficients $\alpha^i_k$ such that $\sum_{i=1}^{N_v} \sum_{k=1}^{l_i} \alpha^i_k \bfa{\psi}^i_k = 1$.  
With the constructed basis functions, their linear span (over $i = 1, \cdots, N_v, k = 1, \cdots l_i$) defines the approximate space $\bfa{V}^m_{\tu{ms}}$ (at the $m$-th inner iteration, to be specified in Section \ref{sec5}).  The GMsFEM solution $\bfa{u}^m_{\tu{ms}} \in \bfa{V}^m_{\tu{ms}}$ can then be obtained via CG global coupling, which is given through the variational form
\begin{equation}\label{ms3}
 a_n(\bfa{u}^m_{\tu{ms}}, \bfa{v}) = (\bfa{f}, \bfa{v})\,, 
 \quad \forall \bfa{v} \in \bfa{V}^m_{\tu{ms}}\,,
\end{equation}
where $a_n$ is defined by (\ref{biform1}).

In summary, one observes that the key component of the GMsFEM is the construction of local basis functions.  First, we will use only the so called offline basis functions, which can be computed in the offline stage.  Second, to improve the accuracy of the multiscale approximation, we will construct additional online basis functions that are problem-dependent and computed locally and adaptively, based on the offline basis functions and some local residuals.  As in \cite{chungres}, our results show that the combination of both offline and online basis functions will give a
rapid convergence of the multiscale solution $\bfa{u}^m_{\tu{ms}}$ to the fine-scale solution $\bfa{u}_h$.

\subsection{Construction of offline multiscale basis functions}
At the current $n$-th Picard iteration, we will present the construction of the offline basis functions (\cite{chungres}).  We start with constructing, for each coarse subdomain $w_i$, a snapshot space $\bfa{V}^i_{\tu{snap}}$.  For simplicity, the index $i$ can be omitted when there is 
no confusion.  The snapshot space $\bfa{V}^i_{\tu{snap}}$ is a set of functions defined on $w_i$ and contains all or most necessary
components of the fine-scale solution restricted to $w_i$.  A spectral problem is then solved in the snapshot space to extract
the dominant modes, which are the offline basis functions and the resulting reduced space is called the offline space. 

\subsubsection{Snapshot space} 
The first choice of $\bfa{V}^i_{\tu{snap}}$ is the restriction
of the conforming space $\bfa{V}_h$ in $w_i$, and the resulting basis functions are called \textbf{spectral basis functions}.  Note that $\bfa{V}^i_{\tu{snap}}$ contains all possible fine-scale functions defined on $w_i$.  

The second choice of $\bfa{V}^i_{\tu{snap}}$ is the set of all $\kappa$-harmonic extensions, and the resulting basis functions are called \textbf{harmonic basis functions}.  More specifically, we denote the fine-grid function $\delta_j^h(\bfa{x}_k): = \delta_{jk}$ for $\bfa{x}_k \in J_h(w_i)$, where $J_h(w_i)$ denotes the set of all nodes of the fine mesh $\mathcal{T}_h$ belonging to $\partial w_i$.  The cardinality of $J_h(w_i)$ is denoted by $J_i$.  At the $n$-th Picard iteration, for each $j=1, \cdots, J_i$, the snapshot function $\bfa{\psi}^i_j$ is defined to be the solution to the following system
\begin{align*}
 -\tu{div}(\kappa(\bfa{x}, |\bfa{Du}_{\tu{ms}}^n|)\bfa{D\psi}^i_j) &= \bfa{0} \quad \text{in } w_i\,,\\
 \bfa{\psi}^i_j &= (\delta^h_j,0) \quad \text{on } \partial w_i\,.
\end{align*}
For each coarse region $w_i$, the corresponding local snapshot space $\bfa{V}^i_{\tu{snap}}$ is defined as $\bfa{V}^i_{\tu{snap}}:= \tu{span} \{\bfa{\psi}^i_j: j=1, \cdots, J_i \}$.  Then, one may define the global snapshot space $\bfa{V}_{\tu{snap}}$ as $\bfa{V}_{\tu{snap}}: = \oplus_{i=1}^{N_v} \bfa{V}^i_{\tu{snap}}$.

For simplicity, in this paper, we will use the first choice of $\bfa{V}^i_{\tu{snap}}$ consisting of the \textbf{spectral basis functions}. 
We also use this choice in our numerical simulations, and still use $J_i$ to denote the number of basis functions of $\bfa{V}^i_{\tu{snap}}$. 

\subsubsection{Offline multiscale basis construction}
To obtain the offline basis functions, we need to perform a space reduction by a spectral problem. The analysis in \cite{chungres} motivates the following construction.  The spectral problem that is needed for the purpose of space reduction is as follows: find $(\bfa{\psi}^i_j, \lambda^i_j) \in \bfa{V}^i_{\tu{snap}} \times \mathbb{R}, j= 1, 2, \cdots, J_i$ such that 
\begin{equation}\label{offb}
 \int_{w_i} \kappa(\bfa{x}, |\bfa{Du}^n_{\tu{ms}}|) \bfa{D\psi}^i_j \cdot \bfa{Dw} \, dx = \lambda^i_j
 \int_{w_i} \tilde{\kappa}(\bfa{x},|\bfa{Du}^n_{\tu{ms}}|) \bfa{\psi}^i_j \cdot \bfa{w} \, dx\,, 
 \quad \forall \bfa{w} \in \bfa{V}^i_{\tu{snap}}\,,
\end{equation}
where the weighted function $\tilde{\kappa}(\bfa{x},|\bfa{Du}^n_{\tu{ms}}|)$ is defined by (see \cite{chungres})
\[\tilde{\kappa}(\bfa{x},|\bfa{Du}^n_{\tu{ms}}|) =  \kappa(\bfa{x}, |\bfa{Du}^n_{\tu{ms}}|) \sum_{i=1}^{N_v} H^2 |\bfa{D} \chi_i|^2\,,\]
and $\{\chi_i\}$ is a set of standard multiscale finite element basis functions, which is a partition of unity, for the coarse node $\bfa{x}_i$ (that is, with linear boundary conditions for cell
problems).  Specifically, $\forall K \in w_i\,,$ $\chi_i$ is defined via 
\begin{align*}
-\ddiv(\kappa(\bfa{x},|\bfa{Du}^n_{\tu{ms}}|)\bfa{\zeta}_i ) 
&= \bfa{0} \textup{ in } K \in w_i\,,\\
\bfa{\zeta}_i &= (\Phi_i,0)^T \textup{ on } \partial K, 
\quad \textup{ for all } K \in w_i \,,\\
\bfa{\zeta}_i &= \bfa{0} \textup{ on } \partial w_i\,,
\end{align*}
where $\Phi_i$ is linear and continuous on $\partial K$.  That is, the multiscale partition of unity is defined as 
$\chi_i = \tilde{\Phi}_i = (\bfa{\zeta}_i)_1$.  

After arranging the eigenvalues $\lambda_j^i$, $j = 1, 2, \cdots, J_i$ from (\ref{offb}) in ascending order, we choose the first $l_i$ eigenfunctions from (\ref{offb}), and denote them by $\bfa{\Psi}_1^{\tu{off}}, \cdots, \bfa{\Psi}_{l_i}^{\tu{off}}$.  Using these eigenfunctions, we can establish the corresponding eigenvectors in the space of snapshots via the formulation
\[\bfa{\phi}_k^{i,\tu{off}} = 
\sum_{j=1}^{J_i} (\Psi_k^{i,\tu{off}})_{j} \bfa{\psi}_j^{i, \tu{snap}}\,,\]
for $k = 1, \cdots , l_i$, where 
$(\Psi_k^{i,\tu{off}})_j$ denotes the $j$-th component of the vector $\bfa{\Psi}_k^{i,\tu{off}}$.  At the final step, the offline basis functions for the coarse neighborhood $w_i$ is defined by $\bfa{\psi}_k^{i, \tu{off}} = \chi_i \bfa{\phi}_k^{i, \tu{off}}$, where $\{\chi_i\}$ is a set of standard multiscale finite element basis functions, which is a partition of unity, for
the coarse neighborhood $w_i$.  We now define the local auxiliary offline multiscale space 
$\bfa{V}_{\tu{off}}^i$ as the linear span of all $\bfa{\psi}_k^{i,\tu{off}}, k = 1,2, \cdots, l_i$.

Using the notation in (\ref{ms3}), one can take $\bfa{V}^m_{\tu{ms}}$ as $\bfa{V}^m_{\tu{off}}:= \tu{span} \{\bfa{\psi}_k^{i,\tu{off}} \, | \, 1 \leq i \leq N_v\,, 1 \leq k \leq l_i \}$.  We refer to \cite{offconverge} for the convergence of the method within the current Picard iteration.

\section{Residual based online adaptive GMsFEM}\label{sec5}
As we mentioned in the previous Sections, some online basis functions are required to obtain a coarse representation of
the fine-scale solution and give a fast convergence of the corresponding adaptive enrichment algorithm.  In \cite{chungres}, such online adaptivity is proposed and mathematically analyzed.  More specifically, at the current $n$-th Picard iteration, when the local residual related to some coarse neighborhood $w_i$ is large (see Subsection \ref{oa}), one may construct a new basis function $\bfa{\phi}_i \in \bfa{V}_i = \bfa{H}^1_0(w_i) \cap \bfa{V}_h$ (with the equipped norm $\| \bfa{v}\|^2_{\bfa{V}_i} = \int_{w_i} \kappa(\bfa{x}, |\bfa{Du}^n_{\tu{ms}}|)|\bfa{Dv}|^2 \, dx$), and add it to the multiscale basis functions space.  It is further shown that if
the offline space contains sufficient information in the form of offline basis functions,
then the online basis construction results in an efficient approximation of the fine-scale
solution $\bfa{u}_h$.

At the considering $n$-th Picard iteration, we use the index $m\text{ } (\geq 1)$ to stand for the adaptive enrichment level.  Thus, $\bfa{V}^m_{\tu{ms}}$ denotes the corresponding GMsFEM space, and $\bfa{u}^m_{\tu{ms}}$ represents the corresponding solution
obtained in (\ref{ms3}).  The sequence of functions $\{\bfa{u}^m_{\tu{ms}}\}_{m\geq 1}$ will converge to the fine-scale solution $\bfa{u}_h$.  In this Section, we remark that our space $\bfa{V}^m_{\tu{ms}}$  can consist of both offline and online basis functions.  We will establish an approach for obtaining the space $\bfa{V}^{m+1}_{\tu{ms}}$
from $\bfa{V}^m_{\tu{ms}}$.

In the following paragraphs, based on \cite{chungres}, we present a framework for the construction of online basis functions.  By online basis functions, we mean basis functions that are computed during the adaptively iterative process, which contrasts with offline basis functions that are computed before the iterative process.  The online basis functions are computed based on some local residuals for the current multiscale solution, that is, the function $\bfa{u}^m_{\tu{ms}}$.  Hence, we realize that some offline basis functions are crucial for the computations of online basis functions.  We will also obtain the sufficiently large number of offline basis functions, which are required in order to get a rapidly converging
sequence of solutions.

At the current $n$-th Picard iteration, we are given a coarse neighborhood $w_i$ and an inner adaptive iteration $m$-th 
with the approximation space $\bfa{V}^m_{\tu{ms}}$.  Recall that the GMsFEM solution $\bfa{u}^m_{\tu{ms}} \in \bfa{V}^m_{\tu{ms}}$ can be obtained by solving (\ref{ms3}):
\begin{align*}
 a_n(\bfa{u}^m_{\tu{ms}}, \bfa{v}) = (\bfa{f}, \bfa{v})\,, 
 \quad \forall \bfa{v} \in \bfa{V}^m_{\tu{ms}}\,.
\end{align*}
Suppose that we need to add a basis function $\bfa{\phi}_i \in \bfa{V}_i$ on the $i$-th coarse neighborhood $w_i$.  Initially, one can set $\bfa{V}^0_{\tu{ms}} = \bfa{V}_{\tu{off}}$.  Let $\bfa{V}^{m+1}_{\tu{ms}} = \bfa{V}^m_{\tu{ms}} + \tu{span}\{\bfa{\phi}_i\}$ be the new approximation space, with $\bfa{u}^{m+1}_{\tu{ms}} \in \bfa{V}^{m+1}_{\tu{ms}}$ being the corresponding GMsFEM solution from (\ref{ms3}).  Let
\[R_i(\bfa{v})= (\bfa{f}, \bfa{v}) - a_n(\bfa{u}^m_{\tu{ms}},\bfa{v})= 
\int_{w_i} \bfa{f} \cdot \bfa{v} \, dx - 
\int_{w_i} \kappa(\bfa{x}, |\bfa{Du}_{\tu{ms}}^n|) \bfa{Du}^m_{\tu{ms}} \cdot \bfa{Dv} \, dx \,, \quad \forall \bfa{v} \in \bfa{V}_i  \,.\]
The argument from \cite{chungres} deduces that the new online basis function $\bfa{\phi}_i \in \bfa{V}_i$ is the solution of 
\begin{equation}\label{ms4}
 a_n(\bfa{\phi}_i, \bfa{v}) = R_i(\bfa{v})\,, \quad \forall \bfa{v} \in \bfa{V}_i\,,
\end{equation}
and $\| \bfa{\phi}_i \|_{\bfa{V}_i} = \| R_i \|_{\bfa{V}^{*}_i}\,.$  This means the residual norm $\|R_i \|_{\bfa{V}^{*}_i}$ (using $\bfa{H}^{-1}(w_i)$ norm) gives a 
measure on the quantity of reduction in energy error.  Also, it holds that 
\[\| \bfa{u}_h -(\bfa{u}^m_{\tu{ms}} + \alpha \bfa{\phi}_i)\|^2_{\bfa{V}_h} = \| \bfa{u}_h - \bfa{u}^m_{\tu{ms}} \|^2_{\bfa{V}_h} - \| \bfa{\phi}_i \|^2_{\bfa{V}_i}\,,\]
for $\alpha = a_n(\bfa{u}_h - \bfa{u}^m_{\tu{ms}}, \bfa{\phi}_i)$.  This algorithm is called the online adaptive GMsFEM because only online basis functions are used.  

The convergence of this algorithm is discussed in \cite{chungres}, within the current Picard iteration.

\subsection{Error estimation in a Picard iteration}
At the current $n$-th Picard iteration, we show a sufficient condition for reduction in the error.  Let $I_p \subset \{1,2, \cdots,N_v\}$ be the index set over coarse neighborhoods $w_i$ $(i \in I_p)$, which are non-overlapping.  For each $i \in I_p$, we define the online basis functions $\bfa{\phi}_i \in \bfa{V}_i$ by the solution to the equation (\ref{ms4}).  Set $\bfa{V}^{m+1}_{\tu{ms}} = \bfa{V}^m_{\tu{ms}} \oplus \tu{span} \{ \bfa{\phi}_i: i \in I_p\}\,.$  Let $r_i = \| R_i \|_{\bfa{V}^{*}_i}$.  Let $\Lambda_p = \displaystyle \min_{ i \in I_p} 
\lambda^i_{l_i +1}$, where $\lambda^i_{l_i +1}$ is the $(l_i+1)$-th eigenvalue ($ k= l_i +1$) from (\ref{offb}) in the coarse region $w_i$.  

From (\cite{chungres}, equation (15)), we obtain the following estimate for the energy norm error reduction
\begin{equation}\label{errp}
 \| \bfa{u}_h -\bfa{u}^{m+1}_{\tu{ms}} \|_{\bfa{V}_h} 
 \leq \left(1 - \dfrac{\Lambda_p \sum_{i \in I_p} r_i^2(\lambda^i_{l_i+1})^{-1}}{C  \sum_{i=1}^{N_v} r_i^2 (\lambda^i_{l_i+1})^{-1}} \right)^{1/2} \| \bfa{u}_h - \bfa{u}^m_{\tu{ms}} \|_{\bfa{V}_h}\,,
\end{equation} 
where $C$ is a uniform constant, independent of the contrast in $\kappa(\bfa{x}, |\bfa{Du}^n_{\tu{ms}}|)$.  Toward error decreasing, for a large enough $\Lambda_p$, there is a need to pick sufficiently many offline basis functions, which means the online error reduction property (\cite{chungres}) as below. 

\begin{definition}
 We say that $\bfa{V}_{\tu{off}}$ satisfies Online Error Reduction Property (ONERP) if
 \[\dfrac{\Lambda_p \sum_{i \in I_p} r_i^2(\lambda^i_{l_i+1})^{-1}}{C  \sum_{i=1}^{N_v} r_i^2 (\lambda^i_{l_i+1})^{-1}} \geq \theta_0\,,\]
for some $\theta_0 > \gamma >0$, where $\gamma$ is independent of physical parameters such as contrast.
\end{definition}

Theoretically, the ONERP is required in order to obtain fast and robust convergence, which is independent of the contrast in the permeability, for general quantities of interest.  

\subsection{Online adaptive algorithm}\label{oa}
Set $m=0$.  Pick a parameter $\theta \in (0,1]$ and denote $\bfa{V}^m_{\tu{ms}} = \bfa{V}^0_{\tu{ms}} = \bfa{V}_{\tu{off}}$.  Choose a small tolerance $\bfa{tol} \in \mathbb{R}_{+}$.  For each $m \in \mathbb{N}$, assume that $\bfa{V}^m_{\tu{ms}}$ is given.  Go to the following \textbf{Step 1}.\\

\noindent \textbf{Step 1:}  Solve for $\bfa{u}^m_{\tu{ms}} \in \bfa{V}^m_{\tu{ms}}$ from the equation (\ref{ms3}).\\

\noindent \textbf{Step 2:}  For each $i=1, \cdots, N_v$, compute the residual $r_i$ for the coarse region $w_i$.  Assume that we obtain
\[r_1 \geq r_2 \geq \cdots \geq r_{N_v}\,.\]

\noindent \textbf{Step 3:}  Pick the smallest integer $k_p$ such that
\[\theta \sum_{i=1}^{N_v} r_i^2 \leq \sum_{i=1}^{k_p} r_i^2\,.\]
Now, for $i=1, \cdots, k_p$, add basis functions $\bfa{\phi}_i$ (by solving (\ref{ms4})) to the space $\bfa{V}^m_{\tu{ms}}$.  The new multiscale basis functions space is defined as $\bfa{V}^{m+1}_{\tu{ms}}$.  That is
\[\bfa{V}^{m+1}_{\tu{ms}} = 
\bfa{V}^m_{\tu{ms}} \oplus \tu{span} \{\bfa{\phi}_i: 1 \leq i \leq k_p\}\,.\]

\noindent \textbf{Step 4:}  If $\displaystyle \sum_{i=1}^{N_v} r_i^2 \leq \bfa{tol}$ or the dimension of $\bfa{V}^{m+1}_{\tu{ms}}$ is sufficiently large, then stop.  Otherwise, set $m \gets m+1$ and go back to \textbf{Step 1}.

\begin{remark}
 If $\theta = 1$, then the adaptive enrichment is said to be uniform.
\end{remark}

In practice, we do not update the basis function space $\bfa{V}^{n+1}_{\tu{ms}}$ ($n \geq 0$) at every Picard iteration step $(n+1)$-th.  We consider a simple 
adaptive strategy to update the basis space. More specifically, at each Picard iteration $(n+1)$-th,
after updating the multiscale solution $\bfa{u}^{n+1}_{\tu{ms}}$ in $\bfa{V}^n_{\tu{ms}}$, we compute the coefficient $\kappa(\boldsymbol{x})=\dfrac{1}{1-\beta(\boldsymbol{x})|(\boldsymbol{Du}^{n+1}_{\tu{ms}})|}$, and the relative $\bfa{L}^2$  change of this updated coefficient and the coefficient corresponding
to the last step that the basis was updated.  If the change is larger than a predefined tolerance, then
we recompute the offline and online basis functions.
In particular, $\delta=0$ implies that we update the basis functions in every Picard iteration,
while $\delta=\infty$ implies that we do not update the basis functions. 
\subsection{GMsFEM for nonlinear elasticity}\label{pu}
We summarize the major  steps of using the GMsFEM to solve problem (\ref{form3}-\ref{D}):
pick a basis update tolerance value $\delta\in \mathbb{R}_{+}$ and Picard iteration
termination tolerance value $\delta_0\in \mathbb{R}_{+}$ (where $\delta_0$ and $\delta$ to be specified in Section \ref{sec6}). We also take an initial guess of $\bfa{u}^{\text{old}}_{\tu{ms}}$, 
and
compute $\kappa^{\text{old}}(\boldsymbol{x})=\dfrac{1}{1-\beta(\boldsymbol{x})|(\bfa{Du}^{\text{old}}_{\tu{ms}})|}$ 
and the multiscale space  $\bfa{V}^{\text{old}}_{\tu{ms}}$, then we repeat following steps:

\noindent \textbf{Step 1:}  
Solve for $\bfa{u}^{\text{new}}_{\tu{ms}} \in \bfa{V}^{\text{old}}_{\tu{ms}}$ from the following equation (as (\ref{ms3})): 
\begin{equation}\label{eold}
 a_{\tu{old}}(\bfa{u}^{\tu{new}}_{\tu{ms}}, \bfa{v}) = (\bfa{f}, \bfa{v}) \quad \forall \bfa{v} \in \bfa{V}^{\text{old}}_{\tu{ms}}\,.
\end{equation}
If $\dfrac{\| \bfa{u}^{\text{new}}_{\tu{ms}}-\bfa{u}^{\text{old}}_{\tu{ms}}\|_{\bfa{V}_h}}{\| \bfa{u}^{\text{old}}_{\tu{ms}}\|_{\bfa{V}_h}}>\delta_0$,
let $\bfa{u}^{\text{old}}_{\tu{ms}}=
\bfa{u}^{\text{new}}_{\tu{ms}}$
and go to \textbf{Step 2}. \\

\noindent \textbf{Step 2:}  Calculate $\kappa^{\text{new}}(\boldsymbol{x})=\dfrac{1}{1-\beta(\boldsymbol{x})|(\bfa{Du}^{\text{new}}_{\tu{ms}})|}$. 
If $\dfrac{||\kappa^{\text{old}}(\boldsymbol{x})-\kappa^{\text{new}}(\boldsymbol{x})||_{\bfa{L}^2(\Omega)}}{||\kappa^{\text{old}}(\boldsymbol{x})||_{\bfa{L}^2(\Omega)}} > \delta$, compute the new basis functions space $\bfa{V}^{\text{new}}_{\tu{ms}}$, 
let $\bfa{V}^{\text{old}}_{\tu{ms}}$=$\bfa{V}^{\text{new}}_{\tu{ms}}$
and $\kappa^{\text{old}}(\boldsymbol{x})=\kappa^{\text{new}}(\boldsymbol{x})$.  Then go to \textbf{Step 1}.


\section{Numerical examples}\label{sec6}
In this section, we will present several test cases to show
the performance of our GMsFEM. 
In our simulations, we consider two choices of $\beta(\boldsymbol{x})$,
which are shown in Figure \ref{fig:model}.
In these two test cases, the blue region represents $\beta(\boldsymbol{x})=1$
and the red regions (channels) represent $\beta(\boldsymbol{x})=10^{-4}$. In addition, the coefficient $\beta(\boldsymbol{x})$ is defined
on a $200\times 200$ fine grid. 
For the coarse grid size, we choose $H=1/20$. We take the source term $f=(\sqrt{x^2+y^2+1},\sqrt{x^2+y^2+1})$,
$\delta_0=10^{-7}$, and $\delta$ as in the following Tables~\ref{tab:m1off}-\ref{tab:new4}.

\begin{figure}[H]
	\centering
	\subfigure[Test model $1$ ($\beta(\boldsymbol{x})$).]{
		\includegraphics[width=2.5in]{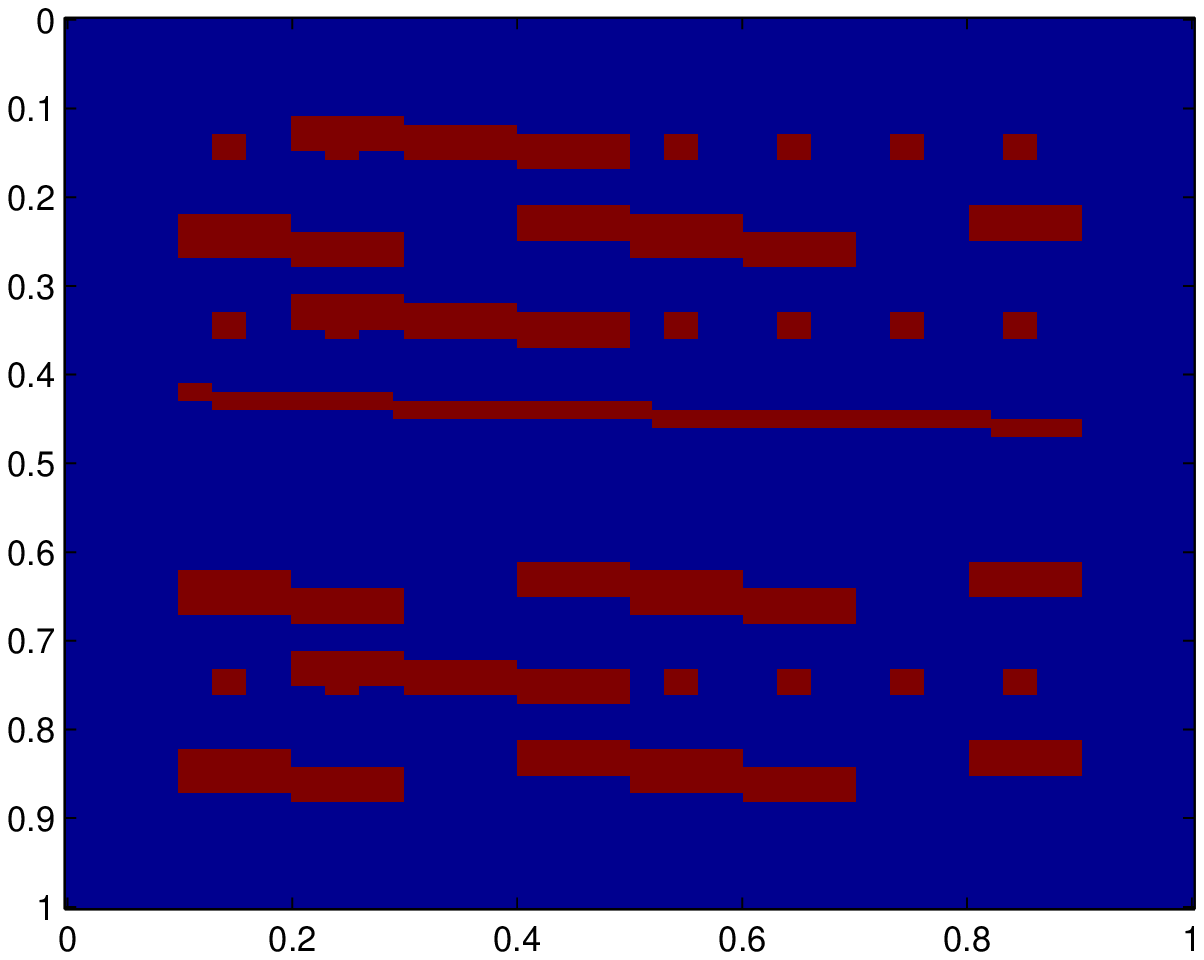}}
	\subfigure[Test model $2$ ($\beta(\boldsymbol{x})$).]{
		\includegraphics[width=2.5in]{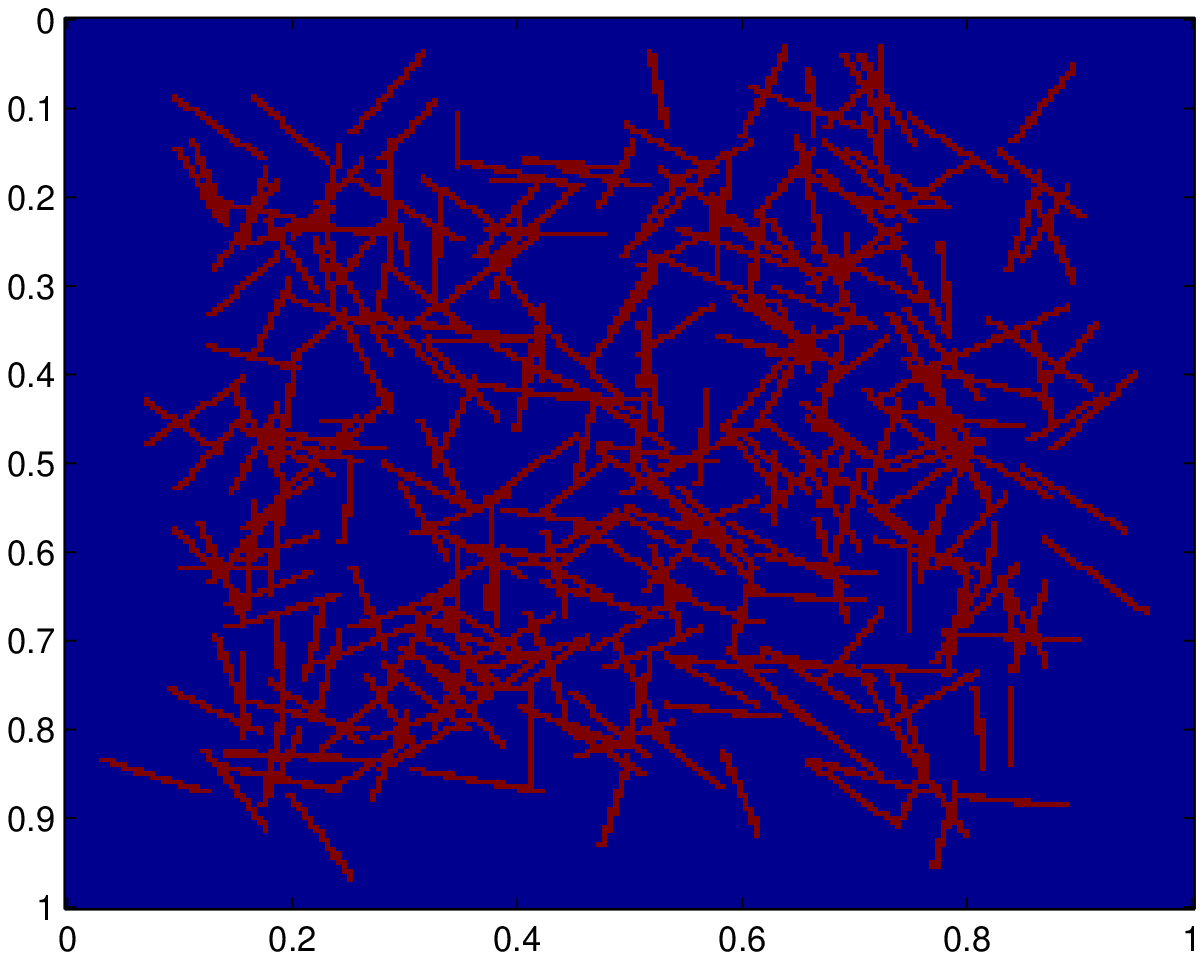}}
	\caption{Coefficients in subdomains (blue: subdomain 1; red: subdomain 2).}
	\label{fig:model} 
\end{figure}

At the $(n+1)$-th Picard iteration, to quantify the quality of our multiscale solutions, 
we use the following
relative  $\bfa{L}^2$ norm error and weighted $\bfa{H}^1$ norm error:
\begin{equation*}
e_{\bfa{L}^2}=\frac{||(\boldsymbol{u_{\tu{ms}}-u_h})||_{\bfa{L}^2(\Omega)}}{||\boldsymbol{u_h}||_{\bfa{L}^2(\Omega)}},\quad
e_{\bfa{H}^1}=\sqrt{\frac{a_n(\boldsymbol{u_{\tu{ms}}-u_h},\boldsymbol{u_{\tu{ms}}-u_h})}
	{a_n(\boldsymbol{u_h},\boldsymbol{u_h})}} 
\end{equation*}
where the reference solution $\boldsymbol{u_h}$ is computed on the fine grid, and the bilinear form $a_n$ is defined in (\ref{biform1}).

We will first present the results using only offline basis functions. We tested the performance
on the use of various number of basis functions, which is denoted by $Nb$. 
The error history with different number of basis functions and update tolerance $\delta$ is shown in Table~\ref{tab:m1off} for the Model 1
and in Table~\ref{tab:m2off} for the Model 2.
First of all, we observe that by updating the basis functions more frequently, 
one can obtain better approximate solutions.
On the other hand, we observe that the errors stay around a fixed level. 
This motivates us to use the online basis functions in order to improve the approximation quality. 

In our next test, we consider the addition of online basis functions. 
In this case, we will construct both offline and online basis functions.
More precisely, we will first find the new offline basis functions by the updated coefficient.
After that, we solve the PDE using the new offline basis functions. Then, using the residual,
we construct online basis functions. 
The error history is shown in Table~\ref{tab:m1on} for the Model 1
and in Table~\ref{tab:m2on} for the Model 2.
In these tables, we use $1+2$ in the $Nb$ column to represent the use of $1$ offline
basis function and $2$ online basis functions. 
From these tables, we observe that updating the basis functions will
produce more significant improvement in the approximate solutions.
Also, we observe that the use of online basis functions is able to
produce more accurate solutions.  

\begin{table}[H]
	\centering
	\begin{adjustbox}{max width=\textwidth}
		\begin{tabular}{|c|c|c|c|c|c|c|c|c|c|c|}
			\hline 
			\multirow{2}{*}{$Nb$} & \multicolumn{2}{c|}{$\delta=\infty$} & \multicolumn{2}{c|}{$\delta=0.5$} & \multicolumn{2}{c|}{$\delta=0.25$}&
			\multicolumn{2}{c|}{$\delta=0.1$}& \multicolumn{2}{c|}{$\delta=0$}\tabularnewline
			\cline{2-11}
			& $e_{L^2}$ & $e_{H^1}$ & $e_{L^2}$ & $e_{H^1} $& $e_{L^2}$ & $e_{H^1} $& $e_{L^2}$ & $e_{H^1} $& $e_{L^2}$ & $e_{H^1} $ \tabularnewline
			\hline
			1&3.992e-02&1.175e-01&3.992e-02&1.175e-01&3.992e-02&1.175e-01&3.992e-02&1.175e-01&3.992e-02&1.175e-01\tabularnewline
			\hline
			3 &9.403e-03& 7.939e-02&9.403e-03&7.939e-02&9.403e-03&7.939e-02&9.403e-03&7.939e-02&9.403e-03& 7.939e-02\tabularnewline
			\hline
			5 &7.754e-03&6.864e-02&7.484e-03&6.769e-02&7.623e-03&6.817e-02&7.574e-03&6.800e-02&7.589e-03&6.805e-02\tabularnewline
			\hline	
			7 &5.815e-03&6.215e-02&3.754e-03&5.401e-02&4.140e-03& 5.589e-02&3.958e-03&5.498e-02&4.007e-03&5.526e-02\tabularnewline
			\hline				
		\end{tabular}
	\end{adjustbox}
	\caption{Results with different update frequency, GMsFEM, $H=1/20$, model 1. }
	\label{tab:m1off}
\end{table}
\begin{table}[H]
	\centering
	\begin{adjustbox}{max width=\textwidth}
		\begin{tabular}{|c|c|c|c|c|c|c|c|c|c|c|}
			\hline 
			\multirow{2}{*}{$Nb$} & \multicolumn{2}{c|}{$\delta=\infty$} & \multicolumn{2}{c|}{$\delta=0.5$} & \multicolumn{2}{c|}{$\delta=0.25$}&
			\multicolumn{2}{c|}{$\delta=0.1$}& \multicolumn{2}{c|}{$\delta=0$}\tabularnewline
			\cline{2-11}
			& $e_{L^2}$ & $e_{H^1}$ & $e_{L^2}$ & $e_{H^1} $& $e_{L^2}$ & $e_{H^1} $& $e_{L^2}$ & $e_{H^1} $ & $e_{L^2}$ & $e_{H^1} $\tabularnewline
			\hline
			1+2&1.140e-02&6.602e-02&4.981e-03&2.884e-02&4.664e-03&1.718e-02&4.612e-03&1.872e-02&4.476e-03&1.163e-02\tabularnewline
			\hline
			3+2 &7.319e-03&5.572e-02&1.197e-03&2.325e-02&6.990e-04&1.006e-02&6.897e-04&1.294e-02&1.030e-06&3.281e-05\tabularnewline
			\hline			
			5+2 &6.512e-03&5.190e-02&7.936e-04&1.708e-02&5.866e-04&9.318e-03&5.238e-04&1.011e-02&2.922e-07&1.300e-05\tabularnewline
			\hline	
			1+4			 &8.429e-03&5.815e-02&1.492e-03&2.348e-02&7.780e-04&1.067e-02&8.478e-04 & 1.360e-02&3.287e-05&2.639e-04\tabularnewline
			\hline				
		\end{tabular}
	\end{adjustbox}
	\caption{Results with different update frequency, GMsFEM, $H=1/20$, model 1. }
	\label{tab:m1on}
\end{table}

\begin{table}[H]
	\centering
	\begin{adjustbox}{max width=\textwidth}
		\begin{tabular}{|c|c|c|c|c|c|c|c|c|c|c|}
			\hline 
			\multirow{2}{*}{$Nb$} & \multicolumn{2}{c|}{$\delta=\infty$} & \multicolumn{2}{c|}{$\delta=0.5$} & \multicolumn{2}{c|}{$\delta=0.25$}&
			\multicolumn{2}{c|}{$\delta=0.1$}& \multicolumn{2}{c|}{$\delta=0$}\tabularnewline
			\cline{2-11}
			& $e_{L^2}$ & $e_{H^1}$ & $e_{L^2}$ & $e_{H^1} $& $e_{L^2}$ & $e_{H^1} $& $e_{L^2}$ & $e_{H^1} $& $e_{L^2}$ & $e_{H^1} $ \tabularnewline
			\hline
			1&4.072e-02&1.228e-01&4.072e-02&1.228e-01&4.072e-02&1.228e-01&4.072e-02&1.228e-01&4.072e-02&1.228e-01\tabularnewline
			\hline
			3 &1.112e-02&9.049e-02&1.112e-02&9.049e-02&1.112e-02&9.049e-02&1.112e-02&9.049e-02&1.112e-02&9.049e-02\tabularnewline
			\hline
			5 &1.001e-02&8.277e-02&9.863e-03&8.209e-02&9.953e-03&8.253e-02&9.922e-03&8.238e-02&9.933e-03&8.244e-02\tabularnewline
			\hline	
			7 &9.501e-03&8.059e-02&7.520e-03&7.173e-02&8.295e-03&7.558e-02&7.979e-03&7.392e-02&8.068e-03&7.443e-02\tabularnewline
			\hline				
		\end{tabular}
	\end{adjustbox}
	\caption{Results with different update frequency, GMsFEM, $H=1/20$, model 2. }
	\label{tab:m2off}
\end{table}

\begin{table}[H]
	\centering
	\begin{adjustbox}{max width=\textwidth}
		\begin{tabular}{|c|c|c|c|c|c|c|c|c|c|c|}
			\hline 
			\multirow{2}{*}{$Nb$} & \multicolumn{2}{c|}{$\delta=\infty$} & \multicolumn{2}{c|}{$\delta=0.5$} & \multicolumn{2}{c|}{$\delta=0.25$}&
			\multicolumn{2}{c|}{$\delta=0.1$}& \multicolumn{2}{c|}{$\delta=0$}\tabularnewline
			\cline{2-11}
			& $e_{L^2}$ & $e_{H^1}$ & $e_{L^2}$ & $e_{H^1} $& $e_{L^2}$ & $e_{H^1} $& $e_{L^2}$ & $e_{H^1} $& $e_{L^2}$ & $e_{H^1} $ \tabularnewline
			\hline
			1+2 &9.917e-03&7.420e-02&4.065e-03&3.159e-02&3.055e-03&2.254e-02&3.171e-03&2.014e-02&2.835e-03&1.186e-02\tabularnewline
			\hline
			3+2 &8.956e-03&7.090e-02&1.350e-03&2.593e-02&1.324e-03&1.738e-02&7.605e-04&1.404e-02&9.314e-07&2.907e-05\tabularnewline
			\hline	
			5+2 & 8.636e-03&6.951e-02&9.361e-04&1.983e-02&1.047e-03&1.565e-02&5.917e-04&1.163e-02&3.039e-07&1.181e-05\tabularnewline
			\hline	
			1+4 &9.130e-03 &7.147e-02&1.673e-03&2.628e-02&1.352e-03&1.744e-02&8.785e-04 & 1.451e-02&4.033e-05&1.504e-04\tabularnewline
			\hline							
		\end{tabular}
	\end{adjustbox}
	\caption{Results with different update frequency, GMsFEM, $H=1/20$, model 2. }
	\label{tab:m2on}
\end{table}

In our second set of simulations, we repeat the above tests
but with $\beta(\boldsymbol{x})=10^{4}$ in the red regions (channels). We take $f=10^{-4}(\sqrt{x^2+y^2+1},\sqrt{x^2+y^2+1})$
to ensure the convergence of the Picard iteration procedure.
From the results in Tables~\ref{tab:new1}-\ref{tab:new4},
we observe similar results to the first set of simulations.  These results indicate that our method
is robust with respect to the contrast in the coefficient, and is able to give accurate approximate solution
with a few local basis functions per each coarse neighborhood.

\begin{table}[H]
	\centering
	\begin{adjustbox}{max width=\textwidth}
		\begin{tabular}{|c|c|c|c|c|c|c|c|c|c|c|}
			\hline 
			\multirow{2}{*}{$Nb$} & \multicolumn{2}{c|}{$\delta=\infty$} & \multicolumn{2}{c|}{$\delta=0.5$} & \multicolumn{2}{c|}{$\delta=0.25$}&
			\multicolumn{2}{c|}{$\delta=0.1$}& \multicolumn{2}{c|}{$\delta=0$}\tabularnewline
			\cline{2-11}
			& $e_{L^2}$ & $e_{H^1}$ & $e_{L^2}$ & $e_{H^1} $& $e_{L^2}$ & $e_{H^1} $& $e_{L^2}$ & $e_{H^1} $& $e_{L^2}$ & $e_{H^1} $ \tabularnewline
			\hline
			1&3.288e-02&1.083e-01&3.288e-02&1.083e-01&3.288e-02&1.083e-01&3.288e-02&1.083e-01&3.288e-02&1.083e-01\tabularnewline
			\hline
			3 &1.051e-02&8.021e-02&1.051e-02&8.021e-02&1.051e-02&8.021e-02&1.051e-02&8.021e-02&1.051e-02&8.021e-02\tabularnewline
			\hline
			5 &8.279e-03&6.707e-02&8.279e-03&6.707e-02&8.279e-03&6.707e-02&8.412e-03&6.739e-02&8.395e-03&6.734e-02\tabularnewline
			\hline	
			7 &6.394e-03&6.218e-02&6.394e-03&6.218e-02&6.394e-03&6.218e-02&5.084e-03&5.906e-02&5.283e-03&5.978e-02\tabularnewline
			\hline				
		\end{tabular}
	\end{adjustbox}
	\caption{Results with different update frequency, GMsFEM, $H=1/20$, model 1, $\beta(\bfa{x})$ value in subdomain 2 is $10^4$. }
	\label{tab:new1}
\end{table}

\begin{table}[H]
	\centering
	\begin{adjustbox}{max width=\textwidth}
		\begin{tabular}{|c|c|c|c|c|c|c|c|c|c|c|}
			\hline 
			\multirow{2}{*}{$Nb$} & \multicolumn{2}{c|}{$\delta=\infty$} & \multicolumn{2}{c|}{$\delta=0.5$} & \multicolumn{2}{c|}{$\delta=0.25$}&
			\multicolumn{2}{c|}{$\delta=0.1$}& \multicolumn{2}{c|}{$\delta=0$}\tabularnewline
			\cline{2-11}
			& $e_{L^2}$ & $e_{H^1}$ & $e_{L^2}$ & $e_{H^1} $& $e_{L^2}$ & $e_{H^1} $& $e_{L^2}$ & $e_{H^1} $& $e_{L^2}$ & $e_{H^1} $ \tabularnewline
			\hline
			1+2&9.329e-03&5.461e-02&9.329e-03&5.461e-02&9.329e-03&5.461e-02&3.265e-03&1.596e-02&2.758e-03&1.153e-02\tabularnewline
			\hline
			3+2 &7.168e-03&4.745e-02&7.168e-03&4.745e-02&7.168e-03&4.745e-02&4.361e-04&7.253e-03&1.218e-06&3.714e-05\tabularnewline
			\hline
			5+2 &6.156e-03&4.414e-02&6.156e-03&4.414e-02&6.156e-03&4.414e-02&3.571e-04&6.565e-03&4.266e-07&1.570e-05\tabularnewline
			\hline	
			1+4 &7.590e-03&4.928e-02&7.590e-03&4.928e-02&7.590e-03&4.928e-02&6.172e-04&7.747e-03&2.060e-05&1.008e-04\tabularnewline
			\hline							
		\end{tabular}
	\end{adjustbox}
	\caption{Results with different update frequency, GMsFEM, $H=1/20$, model 1,
		 $\beta(\bfa{x})$ value in subdomain 2 is $10^4$. }
	\label{}
\end{table}

\begin{table}[H]
	\centering
	\begin{adjustbox}{max width=\textwidth}
		\begin{tabular}{|c|c|c|c|c|c|c|c|c|c|c|}
			\hline 
			\multirow{2}{*}{$Nb$} & \multicolumn{2}{c|}{$\delta=\infty$} & \multicolumn{2}{c|}{$\delta=0.5$} & \multicolumn{2}{c|}{$\delta=0.25$}&
			\multicolumn{2}{c|}{$\delta=0.1$}& \multicolumn{2}{c|}{$\delta=0$}\tabularnewline
			\cline{2-11}
			& $e_{L^2}$ & $e_{H^1}$ & $e_{L^2}$ & $e_{H^1} $& $e_{L^2}$ & $e_{H^1} $& $e_{L^2}$ & $e_{H^1} $& $e_{L^2}$ & $e_{H^1} $ \tabularnewline
			\hline
			1&3.476e-02&1.133e-01&3.476e-02&1.133e-01&3.476e-02&1.133e-01&3.476e-02&1.133e-01&3.476e-02&1.133e-01\tabularnewline
			\hline
			3 &1.678e-02&9.126e-02&1.678e-02&9.126e-02&1.678e-02&9.126e-02&1.678e-02&9.126e-02&1.678e-02&9.126e-02\tabularnewline
			\hline
			5 &1.435e-02&8.065e-02&1.435e-02&8.065e-02&1.435e-02&8.065e-02&1.446e-02&8.065e-02&1.447e-02&8.065e-02\tabularnewline
			\hline	
			7 &1.323e-02&7.853e-02&1.323e-02&7.853e-02&1.323e-02&7.853e-02&1.134e-02&7.853e-02&1.166e-02&7.630e-02\tabularnewline
			\hline			
		\end{tabular}
	\end{adjustbox}
	\caption{Results with different update frequency, GMsFEM, $H=1/20$, model 2, $\beta(\bfa{x})$ value in subdomain 2 is $10^4$. }
	\label{}
\end{table}

\begin{table}[H]
	\centering
	\begin{adjustbox}{max width=\textwidth}
		\begin{tabular}{|c|c|c|c|c|c|c|c|c|c|c|}
			\hline 
			\multirow{2}{*}{$Nb$} & \multicolumn{2}{c|}{$\delta=\infty$} & \multicolumn{2}{c|}{$\delta=0.5$} & \multicolumn{2}{c|}{$\delta=0.25$}&
			\multicolumn{2}{c|}{$\delta=0.1$}& \multicolumn{2}{c|}{$\delta=0$}\tabularnewline
			\cline{2-11}
			& $e_{L^2}$ & $e_{H^1}$ & $e_{L^2}$ & $e_{H^1} $& $e_{L^2}$ & $e_{H^1} $& $e_{L^2}$ & $e_{H^1} $& $e_{L^2}$ & $e_{H^1} $ \tabularnewline
			\hline
			1+2&1.372e-02&6.557e-02&1.372e-02&6.557e-02&1.372e-02&6.557e-02&2.924e-03&1.757e-02&2.279e-03&1.127e-02\tabularnewline
			\hline
			3+2 &1.253e-02&6.323e-02&1.253e-02&6.323e-02&1.253e-02&6.323e-02&5.954e-04&1.034e-02&1.210e-06&3.509e-05\tabularnewline
			\hline
			5+2 &1.161e-02&6.152e-02&1.161e-02&6.152e-02&1.161e-02&6.152e-02&4.242e-04&8.949e-03&3.447e-07&1.271e-05\tabularnewline
			\hline	
			1+4 &1.275e-02&6.371e-02&1.275e-02& 6.371e-02&1.275e-02&6.371e-02&6.797e-04&1.067e-02&8.580e-05&1.352e-04\tabularnewline
			\hline							
		\end{tabular}
	\end{adjustbox}
	\caption{Results with different update frequency, GMsFEM, $H=1/20$, model 2, $\beta(\bfa{x})$ value in subdomain 2 is $10^4$. }
	\label{tab:new4}
\end{table}

\section{Conclusions}\label{sec7}
In this paper, we propose a GMsFEM framework for a strain-limiting nonlinear elasticity model.  The main idea here is the combination of Picard iteration procedure (for linearization) and the two types (offline and online) of basis functions within GMsFEM (for handling the multiple scales and high contrast of materials).  This means that at each Picard iteration, the problem is linear; and in each coarse neighborhood, we use offline multiscale basis functions (whose number is determined by a local error indicator) or combine them with residual based online adaptive basis functions (which are only added in regions with large errors, and can capture global features of the solution).      

Our numerical results show that the combination of offline and online basis functions is able to give accurate solutions, accelerate the convergence, and reduce computational cost with only a small number of Picard iterations as well as basis functions per each coarse region.  In a future contribution, we will address the development of this GMsFEM using the constraint energy minimization approach \cite{cem1,cem2f,cemgle}, for nonlinear problems.

%
%

\vspace{20pt}

\noindent \textbf{Acknowledgments.} 
EC's work is partially supported by Hong Kong RGC General Research Fund (Projects: 14317516 and 14304217)
and CUHK Direct Grant for Research 2017-18.

\bibliographystyle{plain} 

\bibliography{1m_elliptic1}
\end{document}